\documentclass[11pt]{article}
\usepackage[margin=1in]{geometry}
\usepackage{graphicx} % Required for inserting images
\usepackage{amsmath}
\usepackage{amsthm}
\usepackage{amssymb}
\usepackage{bbm}
\usepackage{tikz-cd}
\usepackage[noadjust]{cite}
\usepackage{hyperref}
\usepackage{authblk}

\tikzset{
  symbol/.style={
    draw=none,
    every to/.append style={
      edge node={node [sloped, allow upside down, auto=false]{$#1$}}}
  }
}

\newtheorem{thm}{Theorem}[section]
\newtheorem{prop}[thm]{Proposition}
\newtheorem{lem}[thm]{Lemma}
\newtheorem{cor}[thm]{Corollary}

\newtheorem{thmalpha}{Theorem}

\newcommand{\googlebooks}[1]{(preview at \href{https://books.google.com/books?id=#1}{google books})}

\newcommand{\numdam}[1]{}

\theoremstyle{remark}

\theoremstyle{definition}
\newtheorem{ex}[thm]{Example}
\newtheorem{defn}[thm]{Definition}
\newtheorem{quest}[thm]{Question}

\newtheorem{remark}[thm]{Remark}

\DeclareMathOperator{\ind}{Ind}
\DeclareMathOperator{\Tr}{Tr}
\DeclareMathOperator{\Ad}{Ad}

\DeclareMathOperator{\Fib}{\mathbf{Fib}}
\DeclareMathOperator{\tr}{tr}

\DeclareMathOperator{\End}{End}
\DeclareMathOperator{\QCA}{\mathbf{QCA}}
\DeclareMathOperator{\FDQC}{\mathbf{FDQC}}
\DeclareMathOperator{\coev}{coev}
\DeclareMathOperator{\ev}{ev}

\AtEndDocument{%
  \par
  \noindent
  \medskip
  \begin{tabular}{@{}l@{}}%
    \textsc{Department of Mathematics, North Carolina State University, Raleigh, NC 27695, USA}\\
    \textit{E-mail address}: \texttt{cmjones6@ncsu.edu}
  \end{tabular}

  \par
  \noindent
  \medskip
  \begin{tabular}{@{}l@{}}%
    \textsc{Department of Mathematics, Vanderbilt University, Nashville, TN 37240, USA}\\
    \textit{E-mail address}: \texttt{junhwi.lim@vanderbilt.edu}
  \end{tabular}
  }

\title{An index for quantum cellular automata on fusion spin chains}
\author{Corey Jones, Junhwi Lim}

\date{}

\begin{document}

\maketitle

\begin{abstract}
Interpreting the GNVW index for 1D quantum cellular automata (QCA) in terms of the Jones index for subfactors leads to a generalization of the index defined for QCA on more general abstract spin chains. These include fusion spin chains, which arise as the local operators invariant under a global (categorical/MPO) symmetry, and as the boundary operators of 2D topological codes. We show that for the fusion spin chains built from the fusion category $\mathbf{Fib}$, the index is a complete invariant for the group of QCA modulo finite depth circuits. 
\end{abstract}

\tableofcontents

\section{Introduction}

Quantum cellular automata (QCA) are models of discrete-time unitary dynamics for quantum spin systems \cite{https://doi.org/10.48550/arxiv.quant-ph/0405174}. They have many intriguing connections with a wide array of topics at the intersection of quantum information and condensed matter physics \cite{MR4029266,Farrelly2020reviewofquantum}. Recently, there has been significant interest in studying \textit{topological phases} of QCA as characterized by the group of quantum cellular automata modulo the normal subgroup of finite depth quantum circuits (FDQC) \cite{MR4309221, MR2890305, MR4103966, MR4381173, https://doi.org/10.48550/arxiv.2211.02086,  https://doi.org/10.48550/arxiv.2205.09141,  MR4544190}.\footnote{In some of these settings, other equivalence relations on QCA such as stable equivalence and blending are used.}

In this paper, we are interested in studying quantum cellular automata defined in the context of more general ``abstract" spin systems. Abstract spin systems are defined by nets of local finite-dimensional C*-algebras on a lattice, with the property that operators localized on disjoint regions commute (as in \cite{MR1441540}, see Definition \ref{AbstractSpinChain}). However, unlike ``concrete" spin systems, the local algebras are not required to factorize as tensor products of the algebras localized at individual sites. In this paper, our primary focus will be on \textit{fusion spin chains}, which are nets of algebras on a 1D lattice built from the tensor powers of an object in a fusion category.

Fusion spin chains arise naturally as local operators in a concrete spin chain that are invariant under a global symmetry (group, Hopf \cite{MR1463825}, weak Hopf \cite{WeakHopI, Molnar:2022nmh}, or categorical/MPO \cite{10.21468/SciPostPhys.10.3.053, MR4109480, MR4272039, 2205.15243, MR3719546}), and as boundary operators of 2D topological codes \cite{2307.12552}. There are obvious extensions of the concepts of QCA and FDQC in this context \cite{2304.00068}, which are motivated both by classifying dualities of symmetric local Hamiltonians \cite{Aasen_2016, https://doi.org/10.48550/arxiv.2008.08598, https://doi.org/10.48550/arxiv.2211.03777, PRXQuantum.4.020357} and boundary dynamics of periodic topological systems \cite{PhysRevX.3.031005, PhysRevB.99.085115, aasen2023measurement}, (see Section \ref{motivation} for further discussion). This gives rise to the purely mathematical problem of understanding the group $\QCA/\FDQC$ for a given fusion spin chain.

The GNVW index \cite{MR2890305} provides a complete characterization of QCA on a concrete spin chain up to finite-depth circuits in 1D. Our goal is to consider a generalization of the GNVW index to abstract spin chains by reinterpreting it in terms of the Jones index for subfactors. The Jones index for an inclusion of $\rm{II}_{1}$ factors $N\subseteq M$ is a numerical invariant introduced by V. Jones \cite{MR0696688}. Remarkably, this index is quantized, only taking values in the set $\{4\text{cos}^{2}(\frac{\pi}{n})\ :n\ge 3\}\cup [4,\infty]$. The Jones index is the foundation for the modern theory of subfactors, and has led to the discovery of its many connections with mathematical physics, representation theory, and low-dimensional topology 
\cite{MR0766964, MR2354658, MR1339767, EK23, MR3166042, MR1473221, MR1642584, math.QA/9909027}. It was observed in \cite{Vogts2009DiscreteTQ} that the GNVW index can be expressed as a ratio of Jones indices of subfactors constructed from the spin system algebras. Using a slightly different formulation, we show that this perspective can be used to make sense of the index for QCA on abstract spin chains satisfying what we call the \textit{finite index property}, which in particular holds for fusion spin chains. We demonstrate that the index gives a homomorphism from the group of QCA modulo FDQC to $\mathbbm{R}^{\times}_{+}$.

\begin{thmalpha}\label{indexhom}
    Let $A$ be an abstract spin chain that satisfies the \textit{finite index property} (Definition \ref{FiniteIndexProperty}). Then 
    $$
    \ind: \QCA(A)\rightarrow \mathbbm{R}^{\times}_{+}
    $$
    from Definition \ref{indef} is a homomorphism containing $\FDQC(A)$ in its kernel.

\end{thmalpha}

We then turn to the computation of the index in some concrete examples. We focus our attention on fusion spin chains and introduce a large family of examples of QCA with readily computable indices which we call ``generalized translations." These encapsulate ordinary translations but also more interesting examples that have non-trivial interaction with the DHR category in the sense of \cite{2304.00068}. In many cases of interest, all QCA on fusion spin chains are generalized translations up to composition by a finite depth circuit. 

We then give an in-depth analysis of the group of QCA for the fusion spin chain $A_{\tau}$ built from the rank two fusion category $\textbf{Fib}$ with simple objects $\{\mathbbm{1}, \tau \}$ and fusion rules $\tau^{\otimes 2}\cong \mathbbm{1}\oplus \tau$. This net of algebras has been studied in various physical settings and is closely related to ``the golden anyon chain" (see e.g. \cite{Feiguin:2006ydp, 10.1143/PTPS.176.384, MR3719546}. We show that the kernel of $\ind$ is precisely $\FDQC(A_{\tau})$. In particular, we obtain the following theorem:

\begin{thmalpha}\label{Fibthm}
Let $\phi=\frac{1+\sqrt{5}}{2}$. Then $$\ind:\QCA(A_{\tau})/\FDQC(A_{\tau})\rightarrow \{\phi^{n}\ :\ n\in \mathbbm{Z}\}\subseteq \mathbbm{R}^{\times}_{+}$$ is an isomorphism of groups. In particular, every QCA on $A_{\tau}$ is a composition of a finite depth circuit and a (honest) translation.
\end{thmalpha}

Despite the above theorem, we do not expect $\ind$ to be a complete invariant for fusion spin chains in general. Indeed, in \cite{2304.00068} it is shown that the group $\QCA/\FDQC$ can be non-abelian, even in the case of an abelian global gauge group on a 1D lattice. This implies that the numerical index cannot be complete in general since it is valued in an abelian group. However, we hope that by combining the index and the DHR invariant \cite{2304.00068} we will have a nearly complete picture of the group of topological phases of QCA on abstract spin chains (see Section \ref{Outlook} for further discussion).

\medskip

\textbf{Acknowledgements.} The authors would like to thank Dave Aasen, Dietmar Bisch, Jeongwan Haah, Andrew Schopieray and Dominic Williamson for helpful comments and conversations. The first author is supported by NSF Grant DMS- 2247202. The second author is supported by US ARO grant W911NF2310026.

\section{QCA on abstract spin chains}\label{basicdefn}

In this section, we define abstract spin chains and QCA.

\begin{defn}\label{AbstractSpinChain}
An \textit{abstract spin chain} consists of a unital C*-algebra $A$ and an order homomorphism from the poset $\text{Int}(\mathbbm{Z})$ of finite nonempty intervals in $\mathbbm{Z}$ to the poset of finite dimensional unital subalgebras of $A$, $I\mapsto A_{I}$, such that

\begin{enumerate}
\item
(Locality) If $I\cap J=\varnothing$, then $[A_{I},A_{J}]=0$,
\item
(Quasi-locality) $\bigcup A_{I}$ is norm dense in $A$.

\end{enumerate}

\end{defn}

\begin{remark}
We will typically abuse notation slightly and refer to the entire net by its quasi-local algebra $A$.
\end{remark}

As mentioned in the introduction, there are at least two motivations for considering abstract spin chains. The first is that abstract spin chains arise as nets of algebras invariant under global symmetries (group, Hopf \cite{MR1463825}, weak Hopf \cite{WeakHopI, Molnar:2022nmh}, or categorical/MPO \cite{10.21468/SciPostPhys.10.3.053, MR4109480, MR4272039, 2205.15243, MR3719546})
of concrete spin chains. Suppose we have a spin chain with a finite group of global symmetries acting on site. Define $A^{G}=\{a\in A\ : g(a)=a\ \text{for all}\ g\in G\}$. For each interval $I$, set $A^{G}_{I}= (A_{I})^{G}$. This will give an abstract spin chain which is \textit{not} a concrete spin chain, since the algebras $A^{G}_{I}$ do not split as a tensor product over the algebras localized at sites. This renders many aspects of the analysis of spin chains in the concrete case unworkable. This example is a special case of a more general example:

\begin{ex}\label{FusionCatNets} (Fusion spin chains)
Let $\mathcal{C}$ be a unitary fusion category and $X$ a choice of tensor generating object. We assume $X$ is self-dual and \textit{strongly tensor generating}, meaning there is a positive integer $n$ such that every simple object is a summand of $X^{\otimes n}$. We also assume that $\mathcal{C}$ is strict as a monoidal category (for convenience). Then for any interval $I\subseteq \mathbbm{Z}$, define,
$$A_{I}:= \mathcal{C}(X^{\otimes n}, X^{\otimes n})$$
where $n=|I|$, and for $I=[a,b]\subseteq J=[c,d]$
\begin{eqnarray*}
    &A_{I}\hookrightarrow A_{J}&\\
    \\
    &w\mapsto 1^{\otimes a-c}_{X}\otimes w\otimes 1_X^{\otimes d-b}.&
\end{eqnarray*}
We then take the colimit of these inclusions in the category of C*-algebras to obtain the quasi-local AF C*-algebras

$$A:=\lim_{\text{Int}(\mathbbm{Z})} A_{I}.$$

\noindent For notational convenience, we identify $A_{I}$ with its image in the colimit $A$. Clearly, this data defines an abstract spin chain with the quasi-local algebra $A$. We will denote this net $A(\mathcal{C}, X)$, and call this a \textit{fusion spin chain}.
Fusion spin chains are realized as the local operators in a spin chain invariant under a weak-Hopf \cite{WeakHopI, Molnar:2022nmh} or MPO symmetry \cite{10.21468/SciPostPhys.10.3.053, MR4109480, MR4272039, 2205.15243, MR3719546}.
Alternatively, they arise as the local boundary operators in topologically ordered 2+1 D spin system as in \cite{2307.12552}. Moreover, from a purely mathematical point of view, they arise from the lattice of relative commutants of a finite depth subfactor \cite{MR1642584, KawAnn}.

The self-duality of $X$ in the definition of the fusion spin chain is unnecessary. However, we require $X$ to be self-dual as it facilitates the application of subfactor theory. See Remark \ref{subfactor_perspective} for details.

\end{ex}

In the sequel, for any interval $I=[a,b]\subseteq \mathbbm{Z}$ and $l\ge 0$, we use the notation $I^{+l}=[a-l,b+l]$.

\begin{defn}
    A $*$-automorphism $\alpha\in \text{Aut}(A)$ on an abstract spin chain $A$ is called a \textit{quantum cellular automaton} (QCA) if there exists an integer $l\ge 0$ such that $\alpha(A_I)\subseteq A_{I^{+l}}$ and $\alpha^{-1}(A_{I})\subseteq A_{I^{+l}}$ for all $I\in \text{Int}(\mathbbm{Z})$.\footnote{The condition on $\alpha^{-1}$ follows automatically from the condition on $\alpha$ if we assume some version of Haag duality, see \cite{2304.00068}.} The minimum such $l$ is called the \textit{spread} of $\alpha$.
\end{defn}

It is clear that the composition and the inverse of QCA are QCA, and thus, QCA form a subgroup of $\text{Aut}(A)$. We denote this group $\QCA(A)$.

We now introduce the definition of finite depth circuits. To define a depth one quantum circuit, suppose we have of a partition $\{I_{i}\}_{i\in \mathbbm{N}}$ of $\mathbbm{Z}$ such that $\sup_{i}|I_{i}|=l<\infty$, and for each $i$ a unitary $u_{i}\in A_{I_i}$. From this data, we can construct a QCA on $A$ by defining, for any local operator $w\in A$
$$\alpha(w):=\left(\prod_{i}u_{i}\right)  w \left(\prod_{i} u^{*}_{i}\right).$$
Since any local $w$ commutes with all but finitely many of the $u_{i}$, the above product gives a well-defined $*$-automorphism on the local algebra $\bigcup_{I} A_{I}$. Furthermore, if $w\in A_{I}$, then $\alpha(w), \alpha^{-1}(w)\in A_{I^{+l}}$ (where $l$ is the largest diameter in the partition $\{I_{i}\}$) and hence, extends to a QCA on $A$. We say the QCA $\alpha$ is a \textit{depth one circuit}.

\begin{defn}
    A QCA $\alpha\in \text{Aut}(A)$ is a \textit{finite depth quantum circuit} if it can be written as a composition $\alpha_{1}\circ \alpha_{2}\circ\cdots\circ\alpha_{n}$ where each $\alpha_i$ is a depth one circuit.
\end{defn}

Clearly these form a subgroup of $\QCA(A)$, which we denote $\FDQC(A)$. It is shown in \cite{2304.00068} that $\FDQC(A)$ is a normal subgroup of $\QCA(A)$, which leads us to consider the group $\QCA(A)/\FDQC(A)$. In the setting of concrete spin systems, the group $\QCA(A)/\FDQC(A)$ can be interpreted as characterizing topological phases of discrete unitary dynamics \cite{https://doi.org/10.48550/arxiv.2205.09141}. If we view an abstract net as the observables under a (generalized) global symmetry, then $\QCA(A)/\FDQC(A)$ can be viewed as the topological phases of \textit{symmetric} unitary dynamics. 

\subsection{Physical motivation}\label{motivation}

There are many motivations from both physics and quantum information for studying QCA on concrete spin systems (see, for example, the review article \cite{Farrelly2020reviewofquantum}). In this section, we discuss physical motivations for studying QCA in our more general setting of abstract spin systems (see \cite[Section 2.2]{2304.00068} for additional discussion). 

\begin{enumerate}

\item QCA on spin systems model locally finite-dimensional discrete space/discrete time quantum field theories. In the concrete case, they have been used to approximate continuous quantum field theories (see the review \cite[Section 6.5]{Farrelly2020reviewofquantum}). Concrete spin systems are in some sense the ``topologically trivial" examples, since they have a trivial local superselection theory (i.e. DHR category \cite{2304.00068}). More general abstract spin systems can have non-trivial local superselection sectors if they are obtained, for example, by gauging a global symmetry. These then have the potential to approximate continuum theories with non-trivial superselection sectors which are important, for example, in chiral conformal field theories \cite{Frolich-Gabbiani, Kawahigashi-Longo-Muger}. %This gives a potentially new approach to search for conformal field theories with exotic superselection theories \cite{PhysRevLett.128.231602, PhysRevLett.128.231603}.

\item 
QCA map local Hamiltonians to local Hamiltonians. If the abstract spin system $A^{G}$ consists of local operators invariant under a global symmetry $G$ on the concrete spin system $G$, then $\QCA(A^{G})$ maps symmetric local Hamiltonians to symmetric local Hamiltonians. In some instances, a symmetric QCA mapping between Hamiltonians cannot be extended to a QCA defined on the concrete spin system $A$, hence the equivalence of the two theories 
%need to clarify what we mean by two theories
is only witnessed by taking symmetries into account. This is called \textit{duality}, the most famous example being Kramers-Wannier duality \cite{Aasen_2016, https://doi.org/10.48550/arxiv.2008.08598, https://doi.org/10.48550/arxiv.2211.03777, PRXQuantum.4.020357}. By characterizing symmetric QCA, we parameterize possible dualities of any given Hamiltonians.

\item 
Suppose we have a (n+1)D locally topologically ordered spin system. The study of topological floquet systems \cite{PhysRevB.99.085115, PhysRevX.6.041070} and paths of gapped Hamiltonians \cite{PhysRevB.106.085122} can be approximated by considering finite depth circuit $U$ intertwining the local ground state projections. If we cut a boundary in the spin system such that the finite depth circuit $U^{\prime}$ built from the terms localized on one side of the boundary commutes with the local ground state projections localized on that same side, then conjugation by $U^{\prime}$ induces a QCA on the net of \textit{boundary algebras} of that spin system. This ``chiral boundary dynamics" is then used to characterize the topological order of the bulk circuit.

\end{enumerate}

\noindent We have motivated the study of $\QCA$ on abstract spin systems, but why should we be interested in the quotient group $\QCA/\FDQC$? An immediate practical answer is that $\QCA$ is far too large and unwieldy as a group, and we would have little hope of understanding it. However, a large part of that complexity is contained in the subgroup $\FDQC$ and by quotienting it out, we obtain a manageable group that is amenable to study and classify.

A more principled answer comes from an operational definition of ``topological phases," introduced in \cite{PhysRevB.82.155138}. There it is proposed that two many-body states are in the same topological phase (or more properly, they have the same ``long-range entanglement structure") if there exists a finite depth circuit mapping one to the other. This suggests a natural interpretation of $\QCA/\FDQC$ as the group of topological phases of $\QCA$, as in \cite{https://doi.org/10.48550/arxiv.2205.09141}.

\section{Index for QCA on abstract spin chains}\label{index}

In this section, we will introduce a generalization of the GNVW index \cite{MR2890305} that applies to abstract spin chains with a unique tracial state, satisfying some finite index type properties.

First, recall a \textit{$\rm{II}_{1}$ factor} is an infinite-dimensional von Neumann algebra with trivial center and a normal tracial state. If an infinite dimensional C*-algebra $A$ has a unique tracial state $\tr$, then its bicommutant $A''$ in the GNS representation $L^2(A,\tr)$ is a $\rm{II}_{1}$ factor.
We give a sketch of the proof of this fact: 

Since $A$ has a unique tracial state, $A''$ also has a unique normal tracial state.
Indeed, for $x\in A''$ there is a sequence $\{x_n\}_n\in A$ converging to $x$ in the weak operator topology by the bicommutant theorem.
i.e. for all $a,b\in A$ $\lim\limits_{n\to \infty}\tr(ax_nb)=\tr(axb)$. In particular, when $a=b=1$, $\lim\limits_{n\to \infty}\tr(x_n)=\tr(x)$. It is easy to show that this is a well-defined normal tracial state on $A''$. If $A''$ admits another normal tracial state $\tr_0$, then $\tr_0|_{A}=\tr|_{A}$ by the uniqueness of the tracial state on $A$. Thus, by the normality, $\tr_0=\tr$.

Now, suppose $A''$ is not a factor. Then we can choose any positive element $z$ in the center such that $\tr(z)=1$ and $z\ne1$. Then $\tr(z\cdot(-))$ is another tracial state on $A''$. This contradicts the uniqueness of normal trace on $A''$. Hence, $A''$ is a factor.

For instance, let $A=\overline{\bigcup_{n=0}^\infty A_n}^{\|-\|}$ be an AF C*-algebra such that the inclusion matrix $T_n$ (or the Bratteli diagram) for $A_{n}\subset A_{n+1}$ is indecomposable and $T_n=T_{n-1}^t$ for all $n\ge 1$. Then the trace value of the minimal projections from each of the direct summand of $A_{n+1}$ is determined by the unique Perron-Frobenius eigenvector of $T_n^tT_n$. Hence, $A$ has a unique trace and $A''$ is a factor.

Now, let $M$ be a $\rm{II}_{1}$ factor. Associated to any Hilbert space representation $H$ of $M$ is its \textit{Murray-von Neumann dimension} $\text{dim}_{M}(H)\in \mathbbm{R}_{+}\cup \{\infty\}$. Given an inclusion of $\rm{II}_{1}$ factors $N\subseteq M$, we define the \textit{Jones index} \cite{MR0696688} as
$$
[M:N]:=\text{dim}_{N}(L^{2}(M)).
$$
A subfactor has \textit{finite index} if $[M:N]<\infty$. We recall one of the most useful properties of the Jones index: its multiplicativity. If $N\subseteq P\subseteq M$ is an inclusion of $\rm{II}_{1}$ factors and $[M:N]<\infty$, then $$[M:P], [P:N]<\infty$$ and $$[M:N]=[M:P][P:N].$$ 

\noindent We refer the reader to \cite{MR0696688, MR1473221}
for a proof and various other basic properties of the index.

There is an incredibly rich theory of finite index subfactors \cite{MR0696688,MR3166042}
and our goal is to make use of this theory in the study of QCA defined on abstract spin chains. To this end, we have the following definition, which allows us to make use of the Jones index.

\begin{defn}\label{FiniteIndexProperty}
Let $A$ be an abstract spin chain with a unique tracial state $\tr$, and let $\mathcal{A}:=A^{\prime \prime}\subseteq B(L^{2}(A, \tr))$. $A$ satisfies the (left)
\textit{finite index property} if for any $x\in \mathbbm{Z}$, 

\begin{enumerate}
    \item 
The subalgebra of $\mathcal{A}$, $\mathcal{A}_{x}:=(\bigcup_{I\le x} A_{I})^{\prime \prime}\subseteq B(L^{2}(A,\tr))$, is an infinite-dimensional factor.
\item 
For any $z>x$, the subfactor $\mathcal{A}_{x}\subseteq \mathcal{A}_{z}$ has finite index.
\end{enumerate}

\end{defn}

\begin{remark}
    Since $A$ has a unique tracial state, any QCA preserves $\tr$ and thus extends uniquely to an automorphism of $\mathcal{A}$. Furthermore, it is easy to see from the definitions that if $\alpha\in \QCA(A)$ has spread $l$, then $ \mathcal{A}_{x-l}\subseteq\alpha(\mathcal{A}_{x})\subseteq \mathcal{A}_{x+l}$.
\end{remark}

\begin{defn}\label{indef}
Let $A$ be an abstract spin chain and $\alpha$ be a QCA on $A$ satisfying the finite index property. Let $y\le x$ be integers such that $\mathcal{A}_y\subseteq \alpha(\mathcal{A}_x)$. Then the (left) \textit{index} of $\alpha$ is defined by
$$
\ind(\alpha)=\left(\frac{[\alpha(\mathcal{A}_x):\mathcal{A}_{y}]}{[\mathcal{A}_x:\mathcal{A}_y]}\right)^{1/2}.
$$
The numerator $[\alpha(\mathcal{A}_x):\mathcal{A}_{y}]$ is finite, since there is an integer $z\ge x$ such that $\mathcal{A}_y\subseteq\alpha(\mathcal{A}_x)\subseteq \mathcal{A}_z$ and we have $[\alpha(\mathcal{A}_x):\mathcal{A}_{y}]\le [\mathcal{A}_z:\mathcal{A}_y]<\infty$ by the finite index property.
\end{defn}

\begin{prop}  
$\ind(\alpha)$ is independent of the choice of $x$ and $y$.
\end{prop}
\begin{proof}
We first prove the independence of $y$. Suppose $x$ is fixed and $y$ is the maximum integer that satisfies the condition in Definition \ref{indef}. For $y'\le y\le x$, 
\begin{align*}
    \frac{[\alpha(\mathcal{A}_x):\mathcal{A}_{y'}]}{[\mathcal{A}_x:\mathcal{A}_{y'}]}=\frac{[\alpha(\mathcal{A}_x):\mathcal{A}_{y}][\mathcal{A}_{y}:\mathcal{A}_{y'}]}{[\mathcal{A}_x:\mathcal{A}_y][\mathcal{A}_{y}:\mathcal{A}_{y'}]}=\frac{[\alpha(\mathcal{A}_x):\mathcal{A}_{y}]}{[\mathcal{A}_x:\mathcal{A}_y]}.
\end{align*}
Hence, $\ind(\alpha)$ is independent of $y$. 

Next, we show the independence of $x$. Suppose $x,x', y\in\mathbbm{Z}$ such that $y\le x'\le x$ and $\mathcal{A}_y\subseteq \alpha(\mathcal{A}_{x'})\subseteq\alpha(\mathcal{A}_{x})$. Then
\begin{align*}
    \frac{[\alpha(\mathcal{A}_x):\mathcal{A}_{y}]}{[\mathcal{A}_x:\mathcal{A}_y]}&=\frac{[\alpha(\mathcal{A}_x):\alpha(\mathcal{A}_{x'})][\alpha(\mathcal{A}_{x'}):\mathcal{A}_{y}]}{[\mathcal{A}_x:\mathcal{A}_{x'}][\mathcal{A}_{x'}:\mathcal{A}_y]}=\frac{[\alpha(\mathcal{A}_{x'}):\mathcal{A}_{y}]}{[\mathcal{A}_{x'}:\mathcal{A}_y]}
\end{align*}
The last identity follows from $[\mathcal{A}_x:\mathcal{A}_{x'}]=[\alpha(\mathcal{A}_x):\alpha(\mathcal{A}_{x'})]$.
\end{proof}

\begin{prop} 
Let $\alpha$ and $A$ be as in Definition \ref{indef} and let $y\le x\le z$ be integers satisfying $\mathcal{A}_y\subseteq \alpha(\mathcal{A}_x)\subseteq\mathcal{A}_z$. Then 
\begin{align*}
    \ind(\alpha)=\left(\frac{[\mathcal{A}_z:\mathcal{A}_x]}{[\mathcal{A}_{z}:\alpha(\mathcal{A}_x)]}\right)^{1/2}=\left(\frac{[\mathcal{A}_z:\mathcal{A}_x][\alpha(\mathcal{A}_x):\mathcal{A}_{y}]}{[\mathcal{A}_{z}:\alpha(\mathcal{A}_x)][\mathcal{A}_x:\mathcal{A}_y]}\right)^{1/4}.
\end{align*}
\end{prop}
\begin{proof}
    We first show the first identity.
    \begin{align*}
        \ind(\alpha)^2=\frac{[\alpha(\mathcal{A}_x):\mathcal{A}_{y}]}{[\mathcal{A}_x:\mathcal{A}_y]}=\frac{[\mathcal{A}_z:\mathcal{A}_y]/[\mathcal{A}_{z}:\alpha(\mathcal{A}_x)]}{[\mathcal{A}_z:\mathcal{A}_y]/[\mathcal{A}_{z}:\mathcal{A}_x]}=\frac{[\mathcal{A}_{z}:\mathcal{A}_x]}{[\mathcal{A}_{z}:\alpha(\mathcal{A}_x)]}.
    \end{align*}
    Observe that
    $$
    \ind(\alpha)=\ind(\alpha)^{1/2}\left(\frac{[\mathcal{A}_z:\mathcal{A}_x]}{[\mathcal{A}_{z}:\alpha(\mathcal{A}_x)]}\right)^{1/4}=\left(\frac{[\mathcal{A}_z:\mathcal{A}_x][\alpha(\mathcal{A}_x):\mathcal{A}_{y}]}{[\mathcal{A}_{z}:\alpha(\mathcal{A}_x)][\mathcal{A}_x:\mathcal{A}_y]}\right)^{1/4}.
    $$
    Hence, we have the second identity.
\end{proof}

\begin{prop}
    $\ind(\alpha\circ\beta)=\ind(\alpha)\ind(\beta)$
\end{prop}
\begin{proof}
    Choose integers $y\le y'\le x$ satisfying the following: 
    \begin{enumerate}
        \item $\mathcal{A}_y\subseteq \alpha(\mathcal{A}_{y'})$; 
        \item $\mathcal{A}_{y'}\subseteq \beta(\mathcal{A}_{x})$.
    \end{enumerate}
    Then we have
    \begin{align*}
        \ind(\alpha\circ\beta)^2&=\frac{[\alpha\circ\beta(\mathcal{A}_x):\mathcal{A}_{y}]}{[\mathcal{A}_x:\mathcal{A}_y]}=\frac{[\alpha\circ\beta(\mathcal{A}_x):\alpha(\mathcal{A}_{y'})][\alpha(\mathcal{A}_{y'}):\mathcal{A}_{y}]}{[\mathcal{A}_x:\mathcal{A}_{y'}][\mathcal{A}_{y'}:\mathcal{A}_y]}\\
        &=\frac{[\beta(\mathcal{A}_x):\mathcal{A}_{y'}][\alpha(\mathcal{A}_{y'}):\mathcal{A}_{y}]}{[\mathcal{A}_x:\mathcal{A}_{y'}][\mathcal{A}_{y'}:\mathcal{A}_y]}=\ind(\beta)^2\ind(\alpha)^2.\qedhere
    \end{align*}
\end{proof}

\begin{prop}
    Finite depth circuits are in $\ker(\ind)$.
\end{prop}
\begin{proof}
    It suffices to show that the depth $1$ circuits are in $\ker(\ind)$. Let $\{I\}$ be a collection of intervals that partitions $\mathbbm{Z}$ and $\sup_I|I|<\infty$. Let $\alpha=\Ad(U)$ where $U=\prod_I U_I$ for some unitaries $U_I$ in $\mathcal{A}_I$. Choose an interval $I$ in the partition, and let $x=\sup I$. Then we have $\alpha(\mathcal{A}_x)=\mathcal{A}_x$. Therefore, for any $y\le x$
    \begin{align*}
        \ind(\alpha)^2&=\frac{[\alpha(\mathcal{A}_x):\mathcal{A}_{y}]}{[\mathcal{A}_x:\mathcal{A}_y]}=\frac{[\mathcal{A}_x:\mathcal{A}_{y}]}{[\mathcal{A}_x:\mathcal{A}_y]}=1.\qedhere
    \end{align*}
\end{proof}

\begin{remark} Our treatment of index has been somewhat asymmetric.  We have only assumed that the left infinite algebras are factors, and thus our index only makes sense a-priori with the ``left infinite" algebras. Alternatively, we could consider a (right) finite index property, and define the subfactors $\mathcal{A}^{+}_{x}:=(\bigcup_{I\ge x} A_{I})^{\prime \prime}\subseteq B(L^{2}(A,\tr))$. In this case we could define a ``right" version of the index as follows. For $ x\le z$  with $\mathcal{A}^{+}_{z}\subseteq\alpha(\mathcal{A}^{+}_{x})$, set
$$
\ind(\alpha)=\left(\frac{[\mathcal{A}_x^+:\mathcal{A}_z^+]}{[\alpha(\mathcal{A}_x^+):\mathcal{A}_z^+]}\right)^{1/2}.
$$
In general, this may give different information than the left index (in the case when both indices make sense!). However, in the case of the fusion spin chains, with some work one can show this alternative index is equal to the index defined above. Since we make no use of the ``right" version, we do not include the argument since it would take us too far afield.

\end{remark}

%%%%%%%%%%%%%%%%%%%%%%%%%%%%%%%%%%%

\begin{thm}
If $X\ncong \mathbbm{1}$ is strongly tensor generating, then the net $A(\mathcal{C}, X)$ has a unique tracial state and satisfies the finite index property. Thus, we have a canonical homomorphism

$$\ind: \QCA(A(\mathcal{C},X))/\FDQC(A(\mathcal{C},X))\rightarrow \mathbbm{R}^{\times}_{+}.$$
\end{thm}

\begin{proof}
It suffices to show that the quasi-local algebra has a unique tracial state, and that the left infinite quasi-local algebras are factors in the GNS completion of the global tracial state. 
Choose some $n$ so that $X^{\otimes n}$ contains all isomorphism classes of simple objects in $\mathcal{C}$. 
Set $A_{k}:=A_{(-kn,(k+1)n]}$. Then we can view the quasi-local algebra as the AF-C*-algebra with finite dimensional tower
$$
A_{0}\subseteq A_{1}\subseteq \dots.
$$
The Bratelli diagram (ignoring multiplicities) has at each level vertices indexed by the simple objects of $\mathcal{C}$, and the adjacency matrix given between levels by the fusion graph of $X^{\otimes n} \otimes (\cdot) \otimes X^{\otimes n}$. Thus, the Bratteli diagram is stationary, so there is a unique tracial state given by the Perron-Frobenius eigenvector \cite[Chapter 6]{MR623762}.

Similar reasoning shows that the left algebras $A_{\le x}$ have representations by stationary Bratteli diagrams, hence also have unique tracial state. In particular, the restriction of the global tracial state to this one is unique, and thus, the bicommutant completion in the GNS representation is a $\rm{II}_{1}$ factor (we use $X\ne \mathbbm{1}$ here to guarantee the factor is infinite dimensional).
\end{proof}

\begin{remark}\label{subfactor_perspective}
We can now explain the subfactor perspective on fusion spin chains. Let $A=A(\mathcal{C},X)$ with $X$ self-dual and strongly tensor generating, and let $\mathcal{A}_{x}$ denote the $\rm{II}_{1}$ factors as described above.

Pick any $x$. Then $\mathcal{A}_{x}\subseteq \mathcal{A}_{x+1}$ is a finite index, finite depth subfactor. Indeed, this is the standard model \cite{MR1055708, MR1642584} for the finite depth subfactor with standard invariant

$$
\begin{tikzcd}
 \mathbbm{C}\arrow[r, phantom, "\subseteq"] &\End(X) \arrow[r, phantom, "\subseteq"]& \End(X\otimes \overline{X})\arrow[r, phantom, "\subseteq"] & \End(X\otimes \overline{X}\otimes X)\arrow[r, phantom, "\subseteq"]& \dots
\\
 & \mathbbm{C}\arrow[u, sloped, phantom, "\subseteq"]\arrow[r, phantom, "\subseteq"] &\End(\overline{X})\arrow[r, phantom, "\subseteq"] \arrow[u, sloped, phantom, "\subseteq"] & \End(\overline{X}\otimes X)\arrow[r, phantom, "\subseteq"] \arrow[u, sloped, phantom, "\subseteq"]& \dots
\end{tikzcd}
$$

Since $\overline{X}\cong X$, each of these algebras is isomorphic to $\End(X^{\otimes n})$ in $\mathcal{C}$ for various $n$. Furthermore, $\mathcal{A}_{x}\subseteq \mathcal{A}_{x+1}\subseteq \mathcal{A}_{x+2}$ is an instance of the basic construction, and thus, we view  the entire tower $\dots\subseteq \mathcal{A}_{x-1}\subseteq \mathcal{A}_{x}\subseteq \mathcal{A}_{x+1}\subseteq\cdots$ as the Jones tower of the standard model for the subfactor with the above invariant. 

In particular, this allows us to identify $\End(X^{\otimes z-x})\cong A_{(x,z]}=\mathcal{A}_{x}^{\prime}\cap \mathcal{A}_{z}$ for $x<z$.
This fact is well-known to experts, so we only give a sketch here.
Choose a positive integer $n$ such that $X^{\otimes n}$ contains all simple objects in $\mathcal{C}$ as direct summands. Then for all $k\ge n-1$
$$
\begin{tikzcd}
    A_{[x-k,z]}\arrow[r,symbol=\subset]&A_{[x-k-1,z]}\\
    A_{[x-k,x]}\arrow[r,symbol=\subset]\arrow[u,symbol=\subset]&A_{[x-k-1,z]}\arrow[u,symbol=\subset]
\end{tikzcd}
$$
is a commuting square, and 
$$
\begin{tikzcd}
    A_{[x-k,z]}\arrow[r,symbol=\subset]&A_{[x-k-1,z]}\arrow[r,symbol=\subset]&A_{[x-k-2,z]}\\
    A_{[x-k,x]}\arrow[r,symbol=\subset]\arrow[u,symbol=\subset]&A_{[x-k-1,z]}\arrow[u,symbol=\subset]\arrow[r,symbol=\subset]&A_{[x-k-2,z]}\arrow[u,symbol=\subset]
\end{tikzcd}
$$
is isomorphic to the basic construction of commuting square. Recall that $\mathcal{A}_x=(\bigcup_{k=n-1}^\infty A_{[x-k,x]})''$ and $\mathcal{A}_z=(\bigcup_{k=n-1}^\infty A_{[x-k,z]})''$. Thus, by the Ocneanu compactness, 
$\mathcal{A}_x'\cap \mathcal{A}_z=A_{[x-n,x]}'\cap A_{[x-n+1,z]}$. (See for instance, \cite{MR1473221}.) It is easy to see that $A_{(x,z]}\subset A_{[x-n,x]}'\cap A_{[x-n+1,z]}$. To show the reverse inclusion, let $w\in A_{[x-n,x]}'\cap A_{[x-n+1,z]}$ and $e_{x-n+1}=\frac{1}{\dim(X)}\coev_{X}\circ \ev_X\in A_{[x-n,x-n+1]}$ be the Jones projection. Note that $X$ is self-dual and we assume $\ev_X=\coev_X^*$.
Since $e_{x-n+1}\in A_{[x-n,x-n+1]}\subset A_{[x-n,x]}$ and $w\in A_{[x-n,x]}'$, they must commute. 
Therefore,
\begin{align*}
    w&=[1_X\otimes (\ev_X\circ e_{x-n+1})\otimes 1_X^{\otimes z-x+n-1}]\circ(\coev_X\otimes w)\\
    &=(1_X\otimes \ev_X\otimes 1_X^{\otimes z-x+n-1})\circ(1_X^{\otimes 2}\otimes w)\circ(1_X\otimes e_{x-n+1}\otimes 1_X^{z-x+n-1})\circ(\coev_X\otimes 1_X^{z-x+n})\\
    &=\frac{1}{\dim(X)}1_X\otimes [(\ev_X\otimes 1_X^{\otimes z-x+n-1})\circ(1_X\otimes w)\circ(\coev_X\otimes 1_X^{z-x+n-1})\\
    &\in A_{[x-n+2,z]}.
\end{align*}
Thus, $w\in A_{[x-n+1,x]}'\cap A_{[x-n+2,z]}$. By induction, we obtain that $w\in A_{(x,z]}$.

\end{remark}

\section{Generalized translations on fusion spin chains}

In this section, we consider a class of examples of QCA called \textit{generalized translations}. Consider a fusion spin chain $A(\mathcal{C}, X)$, which in this subsection we denote for $A$ for short. Suppose in addition we have the following data:

\begin{enumerate}
\item 
A fusion category $\mathcal{D}$ and a full inclusion $\mathcal{C}\subseteq \mathcal{D}$.
\item
A factorization $X^{\otimes n}\cong Y\otimes Z$ for some fixed $n$ and $Y,Z\in \mathcal{D}$
 \item 
An isomorphism $\sigma:Y\otimes Z\cong Z\otimes Y$ in $\mathcal{D}$.
\end{enumerate}

We define a QCA $\alpha$ on $A$ as follows. First, partition $\mathbbm{Z}$ into intervals of length $n$. Then consider intervals $I$ which are unions of intervals in the partition. We call these \textit{coarse-grained} intervals. Then for coarse-grained intervals $I$ consisting of $m$ intervals of length $n$,
$$
A_{I}\cong \text{End}_{\mathcal{C}}((X^{\otimes n})^{\otimes m})\cong \text{End}_{\mathcal{D}}((Y\otimes Z)^{\otimes m}).
$$

Recall that for any coarse-grained interval $I$ with $m$ length $n$ intervals, then $I^{+n}$ is again a coarse-grained interval consisting of $m+2$ intervals of length $n$. Then for any coarse-grained interval $I$ and $w\in A_{I}$, we define
$$ 
\alpha(w)=1_{Y\otimes Z}\otimes 1_{Y}\otimes [\sigma^{\otimes m}\circ w\circ (\sigma^{-1})^{\otimes m}]\otimes 1_{Z}\in A_{I^{+n}}
$$
\noindent where $\sigma^{\otimes m}: (Y\otimes Z)^{m}\cong (Z\otimes Y)^{m}$ is simply the $m$-fold tensor product $\sigma \otimes \dots \otimes \sigma$.
This clearly extends to a C*-homomorphism on the quasi-local algebra. Its inverse is given, for $w\in A_{I}$ with $I$ coarse-grained as 
$$ 
\alpha^{-1}(w)= [(\sigma^{-1})^{\otimes m+1}\circ (1_{Z}\otimes w\otimes 1_{Y})\circ \sigma^{\otimes m+1}]\otimes  1_{Y\otimes Z}\in A_{I^{+n}}.
$$

It is easy to see that $\alpha$ and $\alpha^{-1}$ are QCA, which we call \textit{generalized translations}. The reason for this terminology is choosing $\mathcal{D}=\mathcal{C}$, $n=1$ and $X\cong X\otimes 1$ the coherence unitor, then the resulting QCA $\alpha$ is simply translation to the right by one site.
In many cases of interest (including concrete spin systems), generalized translations represent \textit{all possible} QCA up to finite depth circuits, and the index can be used to show this. We have the following computation of the index for these QCA.

\begin{prop}
    Let $\alpha$ be a generalized translation constructed as above from a factorization $X^{\otimes n}\cong Y\otimes Z$. Then $\ind(\alpha)=\dim(Y)$.
\end{prop}

\begin{proof}
Let $x$ be an endpoint of a coarse-grained interval. Then $\mathcal{A}_{x}\subseteq \alpha(\mathcal{A}_{x})$. It is easy to see that $\alpha(\mathcal{A}_{x})$ is precisely the subalgebra of operators in $\mathcal{A}_{x+n}$ which can be written $w\otimes 1_{Z}$. The resulting subfactor $\mathcal{A}_{x}\subseteq \alpha(\mathcal{A}_{x})$ has Jones index $\dim(Y)^{2}$, and thus $\ind(\alpha)=\dim(Y)$.
\end{proof}

\subsection{Concrete spin chains}

Now let $\mathcal{C}=\textbf{Hilb}_{f.d.}$, the category of finite dimensional Hilbert spaces. If we pick $X:=\mathbbm{C}^{d}$, then we obtain the usual definition of a spin chain of qudits. This is the case of QCA studied in \cite{MR2890305}. We can now see that our index agrees with the index defined in \cite{MR2890305}, which we call the \textit{GNVW index}.

\begin{prop}
The definition of index above agrees with the GNVW index for concrete spin chains.
\end{prop}

\begin{proof}
Let us denote the qudit system $A_{d}$.  The GNVW index gives an isomorphism 
$$
\QCA(A_{d})/\FDQC(A_{d})\rightarrow \mathbbm{Z}[\frac{1}{d}]_+^{\times},
$$
where $\mathbbm{Z}[\frac{1}{d}]_+^{\times}$ denotes the group of positive units of the ring inside $\mathbbm{Q}$ generated by $\mathbbm{Z}$ and $\frac{1}{d}$. In particular, this shows that $\QCA(A_{d})/\FDQC(A_{d})$ is generated as a group by (the cosets of) generalized translations, built from a factorization of integers $(\mathbbm{C}^d)^{\otimes k}\cong \mathbbm{C}^{p}\otimes \mathbbm{C}^{q}$ and using the symmetric ``swap" braiding $\sigma$ from $\text{Hilb}_{f.d.}$ for the isomorphism.

%Generalized translations are defined as follows. Pick some factorization $d^{k}=pq$ for some $k>0$. Then partition $\mathbbm{Z}$ into intervals of length $k$, which we order left to right $\dots< I_{i}<I_{i+1}<\dots $. For each $I_{i}$ in the partition, $A_{I_i}=M_{d^k}(\mathbbm{C})\cong M^{i}_{p}(\mathbbm{C})\otimes M^{i}_{q}(\mathbbm{C})$.
Let $I_i=(k(i-1),ki]$ and $x=x_{(1)}\otimes x_{(2)}\in A_{I_i}\cong M_{p}^i(\mathbbm{C})\otimes M_{q}^i(\mathbbm{C})$, we define
$$
\alpha_{p}(x)=(1^{i}_{p}\otimes x_{(2)})\otimes (x_{(1)}\otimes 1^{i+1}_{q})\in A_{I_i}\otimes A_{I_{i+1}}.
$$
%This extends to a QCA on $A_{d}$. 

%Now, for any $\alpha\in \alpha_{p}\cdot \FDQC(A_{d})$, we can easily compute

\noindent By the above proposition, $\ind(\alpha)=p.$ This agrees with the GNVW index, and since the value of the homomorphism $\ind$ agrees with the GNVW index on a generating set, it must be precisely the same homomorphism.
\end{proof}

\subsection{Generalized Kramers-Wannier translations from $G$-graded extensions}

Now consider the case where 
$$
\mathcal{D}=\bigoplus_{g\in G} \mathcal{C}_{g}
$$
is a faithful $G$-graded extension of $\mathcal{C}$, so that $\mathcal{C}_{e}=\mathcal{C}$, each $\mathcal{C}_{g}$ is an invertible $\mathcal{C}$-bimodule, and $\mathcal{C}_{g}\otimes \mathcal{C}_{h}\subseteq \mathcal{C}_{gh}$. We will suppose that $G=\mathbbm{Z}/n\mathbbm{Z}$ with a generator $1$ and the identity $0$, and that we have an
object $Y\in \mathcal{C}_{1}$ such that $X:=Y^{\otimes n}\in \mathcal{C}=\mathcal{C}_{0}$
is a strong tensor generator for $\mathcal{C}$.
Then we have $X=Y\otimes Z$, where $Z=Y^{n-1}$, and we use the isomorphism $\text{id}_{Y^{n}}: Y\otimes Z\cong Z\otimes Y $,
% Do we want a braiding \sigma_{Y,Z} here? 
with $n=1$ in the above construction. 
If we think of the single $X$-sites as being in fact $n$-coarse-grained $Y$ sites, then this is simply the shift to the right by one $Y$ string.

We call a generalized translation of the above type a \textit{generalized Kramers-Wannier duality}, due to the example below.

\begin{remark}
$G$-graded extensions of a fusion category $\mathcal{C}$ are classified by morphisms of 3-groups $\psi: G\rightarrow \text{BrPic}(\mathcal{C})$. Dimensionally reducing and applying the ENO isomorphism (see \cite[Theorem 1.1]{MR2677836}) yields a morphism of 2-groups $\widetilde{\psi}: G\rightarrow \text{Aut}_{br}(\mathcal{Z}(\mathcal{C}))$. For a generalized Kramers-Wannier translation $\alpha$ built from a $\mathbbm{Z}/n\mathbbm{Z}$ extension of $\mathcal{C}$ built from the 3-group $\psi$, let $\xi=\widetilde{\psi}(1)\in \text{Aut}_{br}(\mathcal{Z}(\mathcal{C}))$. Then utilizing the equivalence $\textbf{DHR}(A(\mathcal{C},X))\cong \mathcal{Z}(\mathcal{C})$, it is straightforward to see by unpacking the definitions that $\textbf{DHR}(\alpha)=[\xi]$, where by $[\xi]$ we mean the  monoidal equivalence class of $\xi$.
\end{remark}

\bigskip

\begin{ex}\label{Kramers-Wannier}(Kramers-Wannier translations.) We consider a special case, which extends the original Kramers-Wannier duality. Let $B$ be an abelian group, and consider the unitary fusion category $\mathcal{C}:=\textbf{Hilb}(B)$ of finite dimensional $B$-graded Hilbert spaces. We let $X:=\bigoplus_{g\in B} \mathbbm{C}_{g}$. Then we have $X\cong \mathbbm{C}[B]$, the group algebra viewed as an $B$-graded Hilbert space. 

We consider the fusion spin chain $A(\mathcal{C},X)$. We will construct a QCA on this net using the above $G$-graded recipe.

Now, recall a (unitary Tambara-Yamagami) category with  abelian group of invertibles $B$ is characterized by a non-degenerate, symmetric, bicharacter $\chi$ on $B$ and a choice of sign $\epsilon\in \{-,+\}$ \cite{TAMBARA1998692}. Then $\mathcal{TY}(B,\chi,\epsilon)$ has (isomorphism classes of) simple objects $B\cup \{\rho\}$, with fusion rules 

$$b\otimes a\cong ba$$

\medskip

$$\rho\otimes \rho\cong\bigoplus_{b\in B} b$$

\medskip

$$b\otimes \rho\cong \rho\otimes b\cong \rho$$

\medskip

\noindent Note we have a full inclusion $\textbf{Hilb}(B)\subseteq \mathcal{TY}(B,\chi, \epsilon)$, and in fact $\mathcal{TY}(B,\chi, \epsilon)$ is a $\mathbbm{Z}/2\mathbbm{Z}$-graded extension of $\textbf{Hilb}(B)$. 

Utilizing the unique simple object $\rho\in \mathcal{TY}(B,\chi, \epsilon)$ that is not in the trivially graded component, we have $\rho^{\otimes 2}\cong X\in \textbf{Hilb}(B)$. In particular, we can apply the generalized translation associated to this factorization, and we obtain a $\alpha\in \QCA(A(\mathcal{C},X))$ with 

$$\ind(\alpha)=\sqrt{|B|}.$$

\noindent If we consider the case $B:=\mathbbm{Z}/2\mathbbm{Z}$, this implements a version of the famous Kramers-Wannier duality (for a detailed explanation of this see \cite[Section II.A]{PRXQuantum.4.020357}). We also note that having index values with square roots also appears, in a closely related context, in \cite{PhysRevB.99.085115}.

\end{ex}

\section{Index as an isomorphism for the $\textbf{Fib}$ chain}\label{golden}

%In this section, we consider, we consider the fusion categories $\text{sl}(2,p-2)^{\circ}$, for $p$ an odd prime. These are the subcategories $\text{sl}(2,p-2)$ consisting of integer spin These unitary fusion categories have a unitary

In this section, we consider the fusion spin chain constructed from the fusion category $\Fib$ with simple generating object $\tau$. Recall $\Fib$ is the rank $2$ unitary fusion category with simple objects $\mathbbm{1}$ and $\tau$, with fusion rule $\tau\otimes \tau\cong \mathbbm{1}\oplus \tau$. 

In this section, we denote by $A_\tau$  the fusion spin chain of algebras built from $\Fib$ with generating object $\tau$. Our goal in this section is to use the index to characterize the group $\QCA/\FDQC$ for $A_{\tau}$.

The algebra $A_\tau$ can be viewed either as the boundary algebra 
for a Levin-Wen model built from $\Fib$ (after coarse-graining nearest neighbors), or as the net of local operators on the golden anyon chain, invariant under the $\Fib$ MPO symmetry. From another perspective, this is the even part of the $A_{4}$ subfactor standard invariant, and in fact, the net we build can also be viewed as the two-sided tower of relative commutants in the tunnel of the Jones tower of the unique $A_{4}$ hyperfinite subfactor \cite{MR1642584,MR4272039}.

First, we examine possible values of the index. Note that $\mathcal{A}_{x}\subseteq \mathcal{A}_{x+1}\subseteq \mathcal{A}_{x+2}$ is a basic construction triple, and thus, $\mathcal{A}_{x}\subseteq \mathcal{A}_{x+k}$ is the iterated basic construction, and has index $\phi^{2k}$, where $\phi=\frac{1+\sqrt{5}}{2}$ is the golden ratio.

We recall that the unitary fusion category $\Fib$ is \textit{torsion-free} in the sense of \cite{MR3941472}: every indecomposable Q-system in $\Fib$ is a Morita trivial Q-system. A Q-system is \textit{Morita trivial} if and only if it can be written as $X\otimes \bar{X}$ with multiplication induced from evaluation, for some object $X\in \Fib$ \cite{MR4419534}.

\begin{defn} Let $\mathcal{D}$ be a (full, replete) tensor subcategory of $\text{Bim}(N)$, where $N$ is a $\rm{II}_{1}$ factor. Suppose that $N\subseteq M$ is a finite index subfactor with $L^{2}(M)\in \mathcal{D}$. Then we define the \textit{dual category} of $M$ with respect to $\mathcal{D}$, denoted $\mathcal{D}^{*}_{M}$, as the (full, replete) unitary subcategory of $\text{Bim}(M)$ which is the preimage of $\mathcal{D}$ under the restriction functor $\text{Bim}(M)\rightarrow \text{Bim}(N)$.
\end{defn}

Categorically, $\mathcal{D}^{*}_{M}$ is equivalent to the category of bimodules of the Q-system $_{N} L^{2}(M)_{N}$, internal to $\mathcal{D}$ (see \cite[Proposition 7.11.1]{MR3242743}).

Now, for any $x$ and any $z>x$, the even part of the subfactor $\mathcal{A}_{x}\subseteq \mathcal{A}_{z}$ generates a canonical copy of $\Fib$ inside $\text{Bim}(\mathcal{A}_{x})$ (the same copy of $\Fib$ for all $z$) \cite[Theorem 4.9]{MR1055708}. We call this $\Fib_{x}$.

\begin{lem}\label{FibIntermediate}
  Let $x<z$ and $\mathcal{A}_{x}\subsetneq P\subseteq \mathcal{A}_{z}$ an intermediate subfactor. Then $(\Fib_{x})^{*}_{P}\cong \Fib$, and the $P$-$P$ bimodule $_{P}L^{2}(\mathcal{A}_{z})_{P}\in (\Fib_{x})^{*}_{P}$.

\end{lem}

\begin{proof}
 Since $\mathcal{A}_{x}\subsetneq P\subseteq \mathcal{A}_{z}$, we have $L^{2}(P)\in \Fib_{x}$. Moreover, since $\Fib$ is torsion free, $L^{2}(P)$ is Morita trivial, which means $(\Fib_{x})^{*}_{P}\cong \Fib$. Furthermore, the $P$-$P$ bimodule $L^{2}(\mathcal{A}_{z})$ is in $\Fib_{x}$ as an $\mathcal{A}_{x}$-$\mathcal{A}_{x}$ bimodule by construction, hence $_{P}L^{2}(\mathcal{A}_{z})_{P}\in (\Fib_{x})^{*}_{P}$.
\end{proof}

\begin{thm}
The map $\text{Ind}: \QCA(A_\tau)\rightarrow \mathbbm{R}^{\times}_{+}$ surjects onto the subgroup $\{\phi^{n}\ :\ n\in \mathbbm{Z}\}\cong \mathbbm{Z}$.
\end{thm}

\begin{proof}
First, we claim that the image of $\text{Ind}$ is contained in the subring $\mathbbm{Z}[\phi]$ of $\mathbbm{R}$. Indeed, let $y<x<z$ such that $\mathcal{A}_{y}\subseteq \alpha(\mathcal{A}_{x})\subseteq \mathcal{A}_{z}$. Then $\text{Ind}(\alpha)^{-1}=\left(\frac{[\mathcal{A}_{z}:\alpha(\mathcal{A}_{x})]}{[\mathcal{A}_{z}:\mathcal{A}_{x}]}\right)^{\frac{1}{2}}=[\mathcal{A}_{z}:\alpha(\mathcal{A}_{x})]^{\frac{1}{2}}\phi^{-(z-x)}$. 
By the previous lemma $_{\alpha(\mathcal{A}_{x})} L^{2}(\mathcal{A}_{z}) _{\alpha(\mathcal{A}_{x})}$ is an indecomposable Q-system in $(\Fib_{y})^{*}_{\alpha(\mathcal{A}_{x})}\cong \Fib$, hence there exists some object $X\in (\Fib_{y})^{*}_{\alpha(\mathcal{A}_{x})}$ with $_{\alpha(\mathcal{A}_{x})} L^{2}(\mathcal{A}_{z}) _{\alpha(\mathcal{A}_{x})}\cong X\otimes \bar{X}$.
In particular,
$$
[\mathcal{A}_{z}:\alpha(\mathcal{A}_{x})]=\text{dim}(_{\alpha(\mathcal{A}_{x})} L^{2}(\mathcal{A}_{z}) _{\alpha(\mathcal{A}_{x})})=\text{dim}(X)^{2} 
$$
Thus, $[\mathcal{A}_{z}:\alpha(\mathcal{A}_{x})]^{\frac{1}{2}}=\dim(X)$. Since every object in $\Fib$ has dimension value in $\mathbbm{Z}[\phi]$, we have $\dim(X)\in \mathbbm{Z}[\phi]$. Since $\phi^{-1}=\phi-1$, $\phi^{-(z-x)}=(\phi-1)^{z-x}\in \mathbbm{Z}[\phi]$. Thus, $\text{Ind}(\alpha)^{-1}\in \mathbbm{Z}[\phi]$. Note that $\text{Ind}(\alpha^{-1})^{-1}=\text{Ind}(\alpha)$, so $\text{Ind}(\alpha)\in \mathbbm{Z}[\phi]_+^{\times}$. The group $\mathbbm{Z}[\phi]_+^{\times}$ of positive units of the ring $\mathbbm{Z}[\phi]$ consists precisely of integral powers of $\phi$. Thus, the image of $\ind$ lies in $\{\phi^{n}\ :\ n\in \mathbbm{Z}\}$.

It remains to show that every $\phi^{n}$ is realized by some QCA. These are seen to be the indices of the translations.
\end{proof}

Next, we show that the index $1$ QCA on $A_\tau$ are precisely FDQC. First, given an index $1$ QCA $\alpha$ with spread $l$, we construct a depth $1$-circuit $\beta_1$ such that $\beta_1\circ \alpha$ maps $\mathcal{A}_{2kl}$ onto itself for all $k\in\mathbbm{Z}$. Then, we construct another depth $1$-circuit $\beta_2$ such that $\beta_2\circ\beta_1\circ \alpha$ maps $\mathcal{A}_x$ onto itself for all $x\in\mathbbm{Z}$. Such a QCA automatically fixes the Jones projections. Since the algebra $A_\tau$ is generated by the Jones projections, it follows that $\beta_2\circ \beta_1\circ \alpha$ is the identity. Hence, we have $\alpha=\beta_1^{-1}\circ \beta_2^{-1}$ which is a FDQC.

In order to construct the above depth-1 circuits, one must choose specific local unitaries that constitute these depth-1 circuits. This is done by iterative application of Lemma \ref{pulldown}.

\begin{lem}\label{jonesproj}
    Let $A_0\subseteq A_1\subseteq A_2$ be a unital inclusion of multimatrix algebras and $n^{(i)}$ be the dimension vector of $A_i$. We allow $n^{(i)}$ to have zero entries with the understanding that $A_i\cong \bigoplus_{j: n_j^{(i)}\ne 0}M_{n_j^{(i)}}(\mathbbm{C})$. Suppose we can identify the slots of vectors $n^{(0)}$ and $n^{(2)}$ so that there is an indecomposable matrix $T$ with nonnegative integer entries satisfying 
    $$
    n^{(0)}T^t=n^{(1)}\ \ \ \text{ and }\ \ \ n^{(1)}T=n^{(2)}.
    $$
    Let $\tr$ be a unique tracial state on $A_2$ whose corresponding trace vector $t^{(2)}$ is the Perron-Frobenius eigenvector of $T^t T$ with the eigenvalue $\frac{1}{\lambda}$. Then there is a projection $f$ in $A_0'\cap A_2$ such that $E(f)=\lambda\cdot 1$ where $E$ is the $\tr$-preserving conditional expectation from $A_2$ to $A_1$. 
\end{lem}
\begin{proof}
    Consider the Bratteli diagrams $\Gamma_0^1$, $\Gamma_1^2$, and $\Gamma_0^2$ for the inclusion $A_0\subseteq A_1$, $A_1\subseteq A_2$, and $A_0\subseteq A_1\subseteq A_2$, respectively. We assume each of the edges is directed from the vertex of a smaller algebra to that of a larger algebra.
    By using $\Gamma_0^2$, we can introduce the path algebra representation of $A_0\subseteq A_1\subseteq A_2$. (For the definition of the path algebra, see for instance, \cite{MR1473221}.) For each path $\gamma$ on $\Gamma_0^2$, let $s(\gamma)$ and $r(\gamma)$ be the source and the range of the path $\gamma$, respectively. Since the ``generalized inclusion matrices" $T^t$ and $T$ for $A_0\subseteq A_1$ and $A_1\subseteq A_2$ are transpose to each other, we can pair each edge $\gamma_1$ in $\Gamma_0^1$ with an edge $\gamma_2$ in $\Gamma_1^2$ such that $s(\gamma_1)=r(\gamma_2)$ and $r(\gamma_1)=s(\gamma_2)$. We denote such a $\gamma_2$ by $\gamma_1^*$. Throughout the proof, the edges in $\Gamma_0^1$ and $\Gamma_1^2$ will be denoted with the subscripts $1$ and $2$, respectively, and the length $2$ paths in $\Gamma_0^2$ will be denoted without subscript.
    
    Let $t^{(0)}$ and $t^{(1)}$ be the trace vector of $A_0$ and $A_1$ induced by $\tr$.
    Choose
    $$
    f=\sum_{\genfrac{}{}{0pt}{3}{s(\gamma_1)=s(\delta_1),}{n_{s(\gamma_1)}^{(0)}\ne 0}}\sqrt{\frac{t_{r(\gamma_1)}^{(1)}t_{r(\delta_1)}^{(1)}}{t_{s(\gamma_1)}^{(0)}t_{s(\delta_1)}^{(0)}}}e_{\gamma_1\gamma_1^*,\delta_1\delta_1^*}
    $$
    where $e_{\gamma_1\gamma_2,\delta_1\delta_2}$ are the matrix units in the path algebra representation.
    It is easy to show that $f$ is a projection. For $x\in A_0'\cap A_2$, $x$ can be written as 
    $$
    x=\sum_{{\genfrac{}{}{0pt}{3}{s(\gamma)=s(\delta),}{r(\gamma)=r(\delta)}}}x_{\gamma\delta}e_{\gamma\delta}.
    $$
    By Proposition 5.4.3 of \cite{MR1473221},
    $$
    E(x)=\sum_{{\genfrac{}{}{0pt}{3}{s(\gamma_1)=s(\delta_1),}{r(\gamma_1)=r(\delta_1)}} }\sum_{ \theta_2:s(\theta_2)=r(\gamma_1)}\frac{t_{r(\theta_2)}^{(2)}}{t_{s(\theta_2)}^{(1)}}x_{\gamma_1\theta_2, \delta_1\theta_2} e_{\gamma_1,\delta_1}.
    $$
    where $e_{\gamma_1,\delta_1}=\sum_{\theta_2':s(\theta_2')=r(\gamma_1)}e_{\gamma_1\theta_2', \delta_1\theta_2'}$. Thus,
    \[
        E(f)=\sum_{\gamma_1}\frac{t_{s(\gamma_1)}^{(2)}}{t_{r(\gamma_1)}^{(1)}}\frac{t_{r(\gamma_1)}^{(1)}}{t_{s(\gamma_1)}^{(0)}} e_{\gamma_1,\gamma_1}=\lambda\sum_{\gamma_1} e_{\gamma_1\gamma_1}=\lambda\cdot 1.\qedhere
    \]
\end{proof}

%[Suppose we have the following quadrilateral.
%$$
%\begin{array}{ccc}
 %    M&\subseteq &P  \\
  %   \cup&&\cup\\
   %  N'\cap M&\subseteq& N'\cap P 
%\end{array}
%$$
%Choose $x\in N'\cap P$ and $a\in N$. Then
%$$
%aE_M^P(x)=E_M^P(ax)=E_M^P(xa)=E_M^P(x)a.
%$$
%Thus, $E_M^P(N'\cap P)\subseteq N'\cap M$.
%Hence, the above is a commuting square.]

\begin{lem}\label{MvNequiv}
    Let $p$ and $q$ be projections in $\End(\tau^{\otimes n})$ with the same normalized categorical trace. Then $p$ and $q$ are equivalent. 
\end{lem}
\begin{proof}
   Recall $\text{End}(\tau^{\otimes n})\cong M_{a_{n}}(\mathbbm{C})\oplus M_{b_{n}}(\mathbbm{C})$, where $a_{n}=\text{dim}(\mathcal{C}(\mathbbm{1}, \tau^{\otimes n}))$ and $b_{n}=\text{dim}(\mathcal{C}(\tau, \tau^{\otimes n}))$. The canonical tracial state assigns $\frac{1}{\phi^{n}}$ to minimal projections in the first factor while assigning  $\frac{1}{\phi^{n-1}}$ to minimal projections in the second factor. Thus $\tr(p)=\frac{p_{1}}{\phi^{n}}+\frac{p_{2}}{\phi^{n-1}}$ for integers $0\le p_{1}\le a_{n}$ and $0\le p_{2}\le b_{n}$. 

   If $\tr(p)=\tr(q)$, then we have $$\frac{(p_{1}-q_{1})}{\phi^{n}}+\frac{(p_{2}-q_{2})}{\phi^{n-1}}=0$$
   and multiplying by $\phi^{n}$ we see 
$$(p_1-q_1)+(p_2-q_2)\phi=0.$$
Since $\phi$ is not rational, we must have $p_{1}=q_1$ and $p_{2}=q_{2}$. Thus $p$ is Murray-von Neumann equivalent to $q$, hence there is a unitary conjugating them as desired.
\end{proof}

\begin{lem}\label{pulldown}
    Let $\alpha$ be an index 1 QCA.
    If $\mathcal{A}_{a}\subseteq \alpha(\mathcal{A}_x)\subsetneq\mathcal{A}_{b}$ for $a\le x <b$ then there is a unitary $u\in \mathcal{A}_{a}'\cap \mathcal{A}_{b}=A_{(a,b]}$ such that $\mathcal{A}_{a}\subseteq\Ad(u)(\alpha(\mathcal{A}_x))\subseteq\mathcal{A}_{b-1}$.
\end{lem}
\begin{proof}
    Let $\mathcal{A}_{b-1}\subseteq \mathcal{A}_{b}\stackrel{e_{b+1}}{\subseteq} \mathcal{A}_{b+1}$ be the basic construction. Note that $e_{b+1}\in \mathcal{A}_{b-1}'\cap \mathcal{A}_{b+1}=A_{[b,b+1]}$. Now, by the QCA index $1$ assumption, $[\mathcal{A}_{b}: \alpha(\mathcal{A}_x)]=\phi^{2(b-x)}$. By Lemma \ref{FibIntermediate}, the $\alpha(\mathcal{A}_x)$ bimodule $L^{2}(\mathcal{A}_{b})$ is in $(\Fib_{a})^{*}_{\alpha(\mathcal{A}_x)}\cong \Fib$, and thus there exists some $X\in (\Fib_{a})^{*}_{\alpha(\mathcal{A}_x)}$ with $X\otimes \overline{X}=X\otimes X \cong L^{2}(\mathcal{A}_{b}) $. Here, we used the fact that every object in $\Fib$ is self-dual.
    Note that the function from isomorphism classes of objects of $\Fib$ to $\mathbbm{R}$ given by the dimension function is injective. Thus, 
$$\phi^{2(b-x)}=[\mathcal{A}_{b}: \alpha(\mathcal{A}_x)]=\text{dim}(_{\alpha(\mathcal{A}_{x})} L^{2}(\mathcal{A}_{b}) _{\alpha(\mathcal{A}_{x})}) =\dim(X)^{2}$$
which implies $X\cong \tau^{\otimes b-x}$.

In particular, the standard invariants of $\alpha(\mathcal{A}_x)\subseteq \mathcal{A}_b$ and $\mathcal{A}_x \subseteq \mathcal{A}_b$ are isomorphic. Hence, the adjacency matrix for the principal graph for the subfactor $\alpha(\mathcal{A}_x)\subseteq \mathcal{A}_b$ is 
    $$
    \begin{bmatrix}
    0&1\\
    1&1  
    \end{bmatrix}^{b-x}.
    $$
    Thus, we can choose a subalgebra $\mathcal{B}\subseteq \alpha(\mathcal{A}_x)'\cap \mathcal{A}_b$, whose inclusion matrix in the sense of Lemma \ref{jonesproj} is
    $$
    \begin{bmatrix}
    0&1\\
    1&1  
    \end{bmatrix}.
    $$
    The chain of algebras $\mathcal{B}\subseteq \alpha(\mathcal{A}_x)'\cap \mathcal{A}_b\subseteq \alpha(\mathcal{A}_x)'\cap \mathcal{A}_{b+1}$ satisfies the condition in Lemma \ref{jonesproj}. Thus, there exists a projection $f_{b+1}\in \alpha(\mathcal{A}_x)'\cap \mathcal{A}_{b+1}$ such that $E_{\mathcal{A}_{b}}^{\mathcal{A}_{b+1}}(f_{b+1})=\frac{1}{\phi^2}$ where $E_{\mathcal{A}_{b}}^{\mathcal{A}_{b+1}}$ is the unique trace-preserving conditional expectation from $\mathcal{A}_{b+1}$ to $\mathcal{A}_{b}$. 
    Then by Lemma \ref{MvNequiv}, there is a unitary $v$ in $\mathcal{A}_{a}'\cap \mathcal{A}_{b+1}$ such that $f_{b+1}=v^*e_{b+1}v$. Let $u=\phi^2 E_{\mathcal{A}_{b}}^{\mathcal{A}_{b+1}}(e_{b+1}v)$, i.e. $u$ is the unique element in $\mathcal{A}_{b}$ such that $e_{b+1}v=e_{b+1}u$ by Lemma 1.2 of \cite{MR860811}. Note that $u\in \mathcal{A}_{a}'\cap \mathcal{A}_{b}$. 
    %In fact, $u$ is an element in $\mathcal{A}_{a}'\cap \mathcal{A}_{b}$, since the relative commutants form a commuting square.
    We claim that $u$ is a unitary.
    \begin{align*}
        u^*u&=\phi^2 u^* E_{\mathcal{A}_{b}}^{\mathcal{A}_{b+1}}(e_{b+1})u=\phi^2 E_{\mathcal{A}_{b}}^{\mathcal{A}_{b+1}}(u^*e_{b+1}u)=\phi^2 E_{\mathcal{A}_{b}}^{\mathcal{A}_{b+1}}(v^*e_{b+1}v)\\
        &=\phi^2 E_{\mathcal{A}_{b}}^{\mathcal{A}_{b+1}}(f_{b+1})=1.
    \end{align*}
    Thus, $u$ is an isometry in the finite-dimensional von Neumann algebra $\mathcal{A}_{a}'\cap \mathcal{A}_{b}$. Therefore, it is a unitary. Since $f_{b+1}=u^*e_{b+1}u$ by construction, we have $e_{b+1}=\Ad(u)(f_{b+1})$. Then 
     \begin{align*}
         \mathcal{A}_{a}&=\Ad(u)(\mathcal{A}_{a})\subseteq \Ad(u)(\alpha(\mathcal{A}_{x}))\subseteq \Ad(u)(\mathcal{A}_{b})\cap \{\Ad(u)(f_{b+1})\}'\\
     &=\mathcal{A}_{b}\cap \{e_{b+1}\}'=\mathcal{A}_{b-1}.\qedhere
     \end{align*}
\end{proof}

\begin{cor}\label{fix1}
    Let $\alpha$ be a QCA with the index $1$ and spread $l$. Then for each $x\in\mathbbm{Z}$ there is a unitary $u\in A_{(x-l,x+l]}$ such that $\Ad(u)(\alpha(\mathcal{A}_x))=\mathcal{A}_x$.
\end{cor}
\begin{proof}
    Since the spread of $\alpha$ is $l$, $\mathcal{A}_{x-l}\subseteq \alpha(\mathcal{A}_x)\subseteq\mathcal{A}_{x+l}$ for all $x$.
    Applying Lemma \ref{pulldown} $l$ times gives the desired result.
\end{proof}

\begin{cor}\label{fixmany}
    Let $\alpha$ be a QCA with index $1$ and spread $l$. Then there is a depth-$1$ circuit $\Ad(U)$ such that $\Ad(U)(\alpha(\mathcal{A}_{2kl}))=\mathcal{A}_{2kl}$ for all $k\in\mathbbm{Z}$.
\end{cor}
\begin{proof}
    For each $k\in\mathbbm{Z}$, we can choose a unitary $u_k\in A_{((2k-1)l,(2k+1)l]}$ such that 
    $$
    \Ad(u_k)(\alpha(\mathcal{A}_{2kl}))=\mathcal{A}_{2kl}.
    $$ 
    Since all $u_k$ are supported on disjoint intervals, $\Ad(U)$ for $U=\prod_{k\in\mathbbm{Z}}u_k$ is a depth-1 circuit.
    Fix $k\in\mathbbm{Z}$ and let $U_-=\prod_{j<k}u_j$ and  $U_+=\prod_{j>k}u_j$. 
    Then $U_-$ and $U_+$ are products of mutually commuting unitaries supported on $(-\infty, (2k-1)l]$ and $((2k+1)l,\infty)$.
    Since $\mathcal{A}_{(2k-1)l}\subseteq\alpha(\mathcal{A}_{2kl})\subseteq \mathcal{A}_{(2k+1)l}$,
    \begin{align*}
        \Ad(U_-)(\alpha(\mathcal{A}_{2kl}))&=\alpha(\mathcal{A}_{2kl}),\\
        \Ad(U_+)(\alpha(\mathcal{A}_{2kl}))&=\alpha(\mathcal{A}_{2kl}).
    \end{align*}
    Therefore,
    \begin{align*}
        \Ad(U)(\alpha(\mathcal{A}_{2kl}))&=\Ad(u_k)(\Ad(U_+)(\Ad(U_-)(\alpha(\mathcal{A}_{2kl}))))\\
        &=\Ad(u_k)(\alpha(\mathcal{A}_{2kl}))\\
        &=\mathcal{A}_{2kl}.\qedhere
    \end{align*}
\end{proof}

\begin{cor}\label{fixloc}
    Let $\alpha$ be a QCA with index $1$ and spread $l$. Then there is a depth-$1$ circuit $\Ad(U)$ such that $\Ad(U)(\alpha(A_{(2jl,2kl]}))=A_{(2jl,2kl]}$ for all integers $j<k$.
\end{cor}
\begin{proof}
    Choose a depth-$1$ circuit $U$ as in Corollary \ref{fixmany}. Then we have
    \begin{align*}
        \Ad(U)(\alpha(A_{(2jl,2kl]}))&=\Ad(U)(\alpha(\mathcal{A}_{2jl}'\cap \mathcal{A}_{2kl}))\\
        &=\Ad(U)(\alpha(\mathcal{A}_{2jl}))'\cap \Ad(U)(\alpha(\mathcal{A}_{2kl}))\\
        &=\mathcal{A}_{2jl}'\cap \mathcal{A}_{2kl}\\
        &=A_{(2jl,2kl]}.\qedhere
    \end{align*}
\end{proof}

\begin{thm} If $\alpha$ is a QCA with index 1, then it is a finite depth circuit.
\end{thm}

\begin{proof}
    Let $l$ be the spread of $\alpha$. By Corollary \ref{fixmany} we may assume that $\alpha(\mathcal{A}_{2kl})=\mathcal{A}_{2kl}$ for all $k\in\mathbbm{Z}$. Fix $k\in\mathbbm{Z}$. We claim that for each $0\le n\le 2l-1$ there is $u_n\in A_{(2(k-1)l, 2kl]}$ such that $\Ad(u_n)(\alpha(\mathcal{A}_{2kl-m}))=\mathcal{A}_{2kl-m}$ for all $0\le m\le n$. We use induction on $n$. For $n=0$, we can simply choose $u_0=1$. Suppose we have found $u_{n-1}$. Then 
    \begin{align*}
        \mathcal{A}_{2(k-1)l}&=\Ad(u_{n-1})(\alpha(\mathcal{A}_{2(k-1)l}))\subseteq\Ad(u_{n-1})(\alpha(\mathcal{A}_{2kl-n}))\\
        &\subseteq \Ad(u_{n-1})(\alpha(\mathcal{A}_{2kl-n+1}))=\mathcal{A}_{2kl-n+1}.
    \end{align*}
    By Lemma \ref{pulldown} there is a unitary $v\in A_{(2(k-1)l,2kl-n+1]}$ such that 
    $$
    \mathcal{A}_{2(k-1)l}\subseteq \Ad(v)(\Ad(u_{n-1})(\alpha(\mathcal{A}_{2kl-n})))\subseteq \mathcal{A}_{(2kl-n+1)-1}=\mathcal{A}_{2kl-n}.
    $$
    Since $[\Ad(v)(\Ad(u_{n-1})(\alpha(\mathcal{A}_{2kl-n}))):\mathcal{A}_{2(k-1)l}]=[\mathcal{A}_{2kl-n}:\mathcal{A}_{2(k-1)l}]$, we have 
    $$
    \Ad(v)(\Ad(u_{n-1})(\alpha(\mathcal{A}_{2kl-n})))=\mathcal{A}_{2kl-n}.
    $$ 
    Since $v\in A_{(2(k-1)l,2kl-n+1]}\subseteq \mathcal{A}_{2kl-n+1}$, for $0\le m\le n-1$
    \begin{align*}
        \Ad(v)(\Ad(u_{n-1})(\alpha(\mathcal{A}_{2kl-m})))=\Ad(v)(\mathcal{A}_{2kl-m})=\mathcal{A}_{2kl-m}.
    \end{align*}
    Set $u_n=vu_{n-1}$. This proves the claim. In particular, there is a unitary $w_k=u_{2l-1}$ in $A_{(2(k-1)l, 2kl]}$ such that 
    $$
    \Ad(w_k)(\alpha(\mathcal{A}_x))=\mathcal{A}_x
    $$
    for all $2(k-1)l+1\le x\le 2kl$.
    Take $w_k$ for each $k\in\mathbbm{Z}$ and let $W=\prod_{k\in\mathbbm{Z}}w_k$. Let $x\in\mathbbm{Z}$ and $2(k-1)l+1\le x\le 2kl$ for some $k$. Then by arguing as in Corollary \ref{fixmany} we obtain
    $$
    \Ad(W)(\alpha(\mathcal{A}_x))=\Ad(w_k)(\alpha(\mathcal{A}_x))=\mathcal{A}_x.
    $$
    Thus, $\Ad(W)(\alpha(\mathcal{A}_x))=\mathcal{A}_x$ for all $x\in\mathbbm{Z}$. Arguing as in Corollary \ref{fixloc}, we have that $$
    \Ad(W)(\alpha(A_{[a,b]}))=A_{[a,b]}
    $$
    for all integers $a\le b$. 
    Recall that $A_{[x-1,x]}\cong \End(\tau^{\otimes 2})\cong \mathbbm{C}\oplus\mathbbm{C}$.
    The Jones projection $e_x$ for the inclusion $\mathcal{A}_{x-2}\subseteq \mathcal{A}_{x-1}\stackrel{e_x}{\subseteq}\mathcal{A}_x $ is the unique central projection in $A_{[x-1,x]}$ with tracial state value $1/\phi^2$. Thus, $\Ad(W)(\alpha(e_x))=e_x$. Since any local algebra is generated by the Jones projections, $\Ad(W)\circ\alpha$ is the identity.    
\end{proof}

\section{Toward general topological invariants of QCA}\label{Outlook}

A \textit{topological invariant} of QCA is a group $G$ and a homomorphism $\pi: \QCA(A)\rightarrow G$ with $\FDQC(A)\subseteq \text{ker}(\pi)$. Ideally $G$ should be some sort of ``well-understood" group and the homomorphism $\pi$ should be easily computable.

One approach to the classification of topological phases of $\QCA$ in terms of invariants is to find a \textit{complete} set of topological invariants. A finite family of topological invariants $\pi_{i}: \QCA\rightarrow G_{i}$ is \textit{complete} if  $\bigcap_i \text{ker}(\pi_{i})=\FDQC(A)$. Then the homomorphism

\begin{eqnarray*}
    &\prod_i \pi_{i}:\QCA(A)\rightarrow \prod_i G_{i}&\\
    \\
    &\alpha\mapsto (\pi_{1}(\alpha), \dots, \pi_{n}(\alpha))&
\end{eqnarray*}
induces an isomorphism of $\QCA(A)/\FDQC(A)$ onto its image.

The two examples of topological invariants we know of that apply to $\QCA$ on fusion spin chains are $\ind$ as developed here and the homomorphism $\textbf{DHR}: \QCA(A)\rightarrow \text{Aut}_{br}(\textbf{DHR}(A))$ from \cite{2304.00068}. This leads to the following question:

\begin{quest}
Let $A$ be a fusion spin chain. Is the pair $(\ind, \textbf{DHR})$ a complete set of topological invariants? If not, what are other topological invariants?
\end{quest}

Our analysis shows the pair is complete for the fusion category $\Fib$. Indeed, we show $\ind$ itself is a complete invariant. Moreover, $\text{Aut}_{br}(\mathcal{Z}(\Fib))$ is trivial. In general, however, it appears that $\ind$ and $\textbf{DHR}$ are not completely independent, as we can see from examining the generalized Kramers-Wannier translations from Example \ref{Kramers-Wannier}.

\nocite{*}
\bibliographystyle{alpha}
%\DeclareRobustCommand{\disambiguate}[3]{#2~#3}
\bibliography{Reference}{}
\end{document}